\documentclass[12pt]{amsart}

\usepackage{url,amssymb,amsmath,amsthm,amscd,paralist}
\usepackage{tkz-graph}
\usepackage{tikz-cd}
\usepackage{algpseudocode}
\usepackage[h]{esvect}
\usepackage{fancyhdr}
\usepackage{amssymb}
\usepackage{graphicx}
\usepackage[T1]{fontenc}
\usepackage{fullpage}
\usepackage{enumitem}
\usepackage{comment}
\usepackage{soul}

\usepackage{tikz-cd}
\usetikzlibrary{positioning}
\usetikzlibrary{arrows}
\usetikzlibrary{patterns}
\usetikzlibrary{calc,backgrounds}

\newcommand{\midarrow}{\tikz \draw[- triangle 90] (0,0) -- +(.1,0);}
\usetikzlibrary{decorations.pathreplacing,decorations.markings}
\tikzset{arrow data/.style 2 args={%
		decoration={%
			markings,
			mark=at position #1 with \arrow{#2}},
		postaction=decorate}
}%
\usepackage{graphicx}
\tikzstyle{vertex}=[circle, draw, inner sep=0pt, minimum size=6pt]
\tikzstyle{rvertex}=[circle, red, fill, draw, inner sep=0pt, minimum size=6pt]
\tikzstyle{gvertex}=[circle, green, fill, draw, inner sep=0pt, minimum size=6pt]
\tikzstyle{bvertex}=[circle, blue, fill, draw, inner sep=0pt, minimum size=6pt]
\tikzstyle{Bvertex}=[circle, black, fill, draw, inner sep=0pt, minimum size=6pt]
\tikzstyle{pvertex}=[circle, purp, fill, draw, inner sep=0pt, minimum size=6pt]
\tikzstyle{overtex}=[circle, orange, fill, draw, inner sep=0pt, minimum size=6pt]

\newcommand{\Bvertex}{\node[Bvertex]}

\usepackage{float}
\usepackage[font={small},labelfont={bf}]{caption}
\usepackage{subcaption}

\usepackage{rotating}

\usepackage{multicol}

\tikzstyle{vertex}=[circle,black, fill=black, draw, inner sep=0pt, minimum size=6pt]
\usepackage{tikz}
\usetikzlibrary{decorations.pathreplacing}
\usepackage{xcolor}
\definecolor{cof}{RGB}{219,144,71}
\definecolor{pur}{RGB}{186,146,162}
\definecolor{greeo}{RGB}{91,173,69}
\definecolor{greet}{RGB}{52,111,72}

\newtheorem{thm}{Theorem}
\newtheorem{lem}[thm]{Lemma}
\newtheorem{cor}[thm]{Corollary}

\newtheorem{claim}[thm]{Claim}

\theoremstyle{definition}

\theoremstyle{remark}

\begin{document}


\author{Ryan Alvarado}
\address[Alvarado]{Department of Mathematics and Statistics, Amherst College, Amherst, MA 01002, USA}
\email{rjalvarado@amherst.edu}

\author{Maia Averett}
\address[Averett]{Department of Mathematics and Computer Science, Mills College, Oakland, CA 94613, USA}
\email{maverett@mills.edu}

\author{Benjamin Gaines}
\address[Gaines]{Department of Mathematics and Physics, Iona College, New Rochelle, NY 10801, USA}
\email{bgaines@iona.edu}

\author{Christopher Jackson}
\address[Jackson]{59433-019, FCI Coleman Low, 
Federal Correctional Institution,
P.O. Box 1031
Coleman, FL 33521}

\author{Mary Leah Karker}
\address[Karker]{Department of Mathematics and Computer Science, Providence College, Providence, RI 02918, USA}
\email{mkarker@providence.edu}

\author{Malgorzata Aneta Marciniak}
\address[Marciniak]{LaGuardia Community College of the City University of New York, Long Island City, NY 11101, USA}
\email{mmarciniak@lagcc.cuny.edu}

\author{Francis Su}
\address[Su]{Department of Mathematics, Harvey Mudd College, Claremont, CA 91711, USA}
\email{su@math.hmc.edu}

\author{Shanise Walker}
\address[Walker]{Department of Mathematics, University of Wisconsin-Eau Claire, Eau Claire, WI 54701, USA}
\email{walkersg@uwec.edu}



\thanks{This work was conducted as part of the Research Experiences for Undergraduate Faculty (REUF) Workshop. REUF is a program of the American Institute of Mathematics (AIM) and the Institute for Computational and Experimental Mathematics (ICERM), made possible by the support from the National Science Foundation (NSF) through Grants NSF-DMS 1620073 to AIM and NSF-DMS 1620080 to ICERM}

\title{The Game of Cycles}

\begin{abstract}
The Game of Cycles, introduced by Su (2020), is played on a simple connected planar graph together with its bounded cells, and players take turns marking edges with arrows according to a sink-source rule that gives the game a topological flavor. The object of the game is to produce a \emph{cycle cell}---a cell surrounded by arrows all cycling in one direction---or to make the last possible move. We analyze the two-player game for various classes of graphs and determine who has a winning strategy. We also establish a topological property of the game:\ that a board with every edge marked must have a cycle cell.
\end{abstract}
\maketitle

    

\section{The Game of Cycles}

A fertile mix of topological and graph-theoretic questions arise from considering the following game, called the Game of Cycles.  

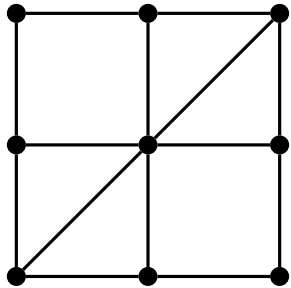
\begin{figure}[H]
	\centering
\begin{tikzpicture}[scale=1.75]
			 \begin{scope}[very thick, every node/.style={sloped,allow upside down}]
			 \Bvertex (A) at (-1,1){};
			 \Bvertex (B) at (0,1){};
			 \Bvertex (C) at (1,1){};
			 \Bvertex (D) at (-1,0){};
			 \Bvertex (E) at (0,0){};
			 \Bvertex (F) at (1,0){};
			 \Bvertex (G) at (-1,-1){};
			 \Bvertex (H) at (0,-1){};
			 \Bvertex (I) at (1,-1){};
			 \draw (A) to (B);
			 \draw (A) to (D);
			 \draw (B) to (C);
			 \draw (B) to (E);
			 \draw (C) to (E);
			 \draw (C) to (F);
			 \draw (D) to (E);
			 \draw (D) to (G);
			 \draw (E) to (F);
			 \draw (E) to (H);
			 \draw (F) to (I);
			 \draw (G) to (E);
			 \draw (G) to (H);
			 \draw (H) to (I);
			 \end{scope}
\end{tikzpicture}
	\caption{A board is a connected planar graph together with its bounded cells.}
	\label{gameboard}
\end{figure}


Start with any simple connected planar graph of dots (vertices) and edges. It divides a bounded region of the plane into \emph{cells}. A graph together with its bounded cells is a \emph{game board}. Two players take turns marking one unmarked edge with an arrow pointing along the edge in one direction or the other.  The arrows must obey a \emph{sink-source rule}:\ players are not allowed to create a \emph{sink} (a dot all of whose edges are all marked pointing toward that dot) or a \emph{source} (a dot all of whose edges are marked pointing away from that dot). Each edge can admit only one arrow, and arrows serve the same function in the game no matter who marks them.
Players must make a move if they have a move available.

\begin{figure}[H]
	\centering
	{
\begin{tikzpicture}[scale=1.75]
			 \begin{scope}[very thick, every node/.style={sloped,allow upside down}]
			 \Bvertex (A) at (-1,1) [label=above left:$a$]{};
			 \Bvertex (B) at (0,1) [label=above left:$b$]{};
			 \Bvertex (C) at (1,1) [label=above left :$c$]{};
			 \Bvertex (D) at (-1,0) [label=above left:$d$]{};
			 \Bvertex (E) at (0,0) [label=above left:$e$]{};
			 \Bvertex (F) at (1,0) [label=above left:$f$]{};
			 \Bvertex (G) at (-1,-1) [label=above left:$g$]{};
			 \Bvertex (H) at (0,-1) [label=above left:$h$]{};
			 \Bvertex (I) at (1,-1) [label=above left:$i$]{};
			 \draw (B) -- node {\midarrow} (A);
			 \draw (D) to (A);
			 \draw (B) -- node {\midarrow} (C);
			 \draw (B) -- node {\midarrow} (E);
			 \draw (E) to (C);
			 \draw (C) to (F);
			 \draw (E) -- node {\midarrow} (D);
			 \draw (D) -- node {\midarrow} (G);
			 \draw (E) to (F);
			 \draw (H) -- node {\midarrow} (E);
			 \draw (F) -- node {\midarrow} (I);
			 \draw (H) -- node {\midarrow} (I);
			 \draw (E) to (G);
			 \draw (G) -- node {\midarrow} (H);
			 \end{scope}
\end{tikzpicture}
}
	\hskip.7in
	%
	{
\begin{tikzpicture}[scale=1.75]
			 \begin{scope}[very thick, every node/.style={sloped,allow upside down}]
			 \Bvertex (A) at (-1,1) [label=above left:$a$]{};
			 \Bvertex (B) at (0,1) [label=above left:$b$]{};
			 \Bvertex (C) at (1,1) [label=above left :$c$]{};
			 \Bvertex (D) at (-1,0) [label=above left:$d$]{};
			 \Bvertex (E) at (0,0) [label=above left:$e$]{};
			 \Bvertex (F) at (1,0) [label=above left:$f$]{};
			 \Bvertex (G) at (-1,-1) [label=above left:$g$]{};
			 \Bvertex (H) at (0,-1) [label=above left:$h$]{};
			 \Bvertex (I) at (1,-1) [label=above left:$i$]{};
			 \draw (A) -- node {\midarrow} (B);
			 \draw (B) -- node {\midarrow} (E);
			 \draw (E) -- node {\midarrow} (D);
			 \draw (D) -- node {\midarrow} (A);
			 \draw (B) to (C);
			 \draw (C) to (E);
			 \draw (F) to (C);
			 \draw (E) to (F);
			 \draw (D) to (E);
			 \draw (G) to (D);
			 \draw (H) to (E);
			 \draw (I) to (F);
			 \draw (E) to (G);
			 \draw (G) to (H);
			 \draw (H) to (I);
			 \end{scope}
\end{tikzpicture}
}
	\caption{The board on the left has a source at $b$ and a sink at $i$. These are not allowed in the Game of Cycles. Also, the square $edgh$ is not a cycle cell, because the path $e \to d \to g \to h \to e$ does not enclose a single cell. In the board on the right, the square $dabe$ is a cycle cell.}
	\label{gamerules}
\end{figure}
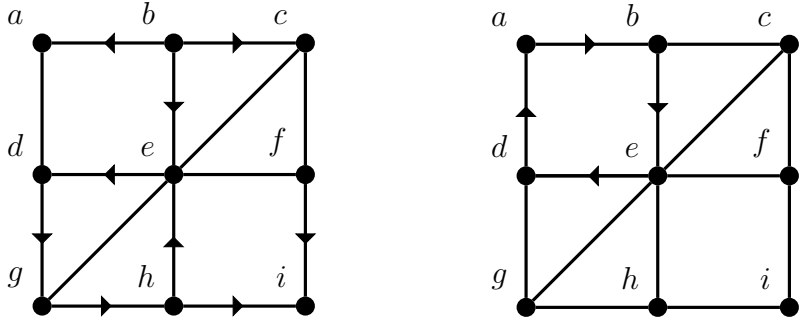

The object of the game is to produce a \emph{cycle cell}, a single cell in the board whose boundary edges are all marked by arrows all cycling in the same direction (either clockwise or counterclockwise). The first person to create a cycle cell wins the game, but if play ends without a cycle cell, the person who makes the last possible move is declared the winner.


The Game of Cycles was introduced in Francis Su's recent book \emph{Mathematics for Human Flourishing} \cite{su2020}, along with two natural questions:\ (1) Who has a winning strategy in this game and what is that strategy? (2) If every edge of a board is marked, must there be a cycle cell? In this paper we shall investigate (1) for various classes of game boards, answer (2) in the affirmative, and suggest several more questions suitable for exploration and play. 

\section{Game Play}

Play this game for a while and you'll notice some interesting things.

Marking edges has important consequences for nearby edges. For example, if you mark an edge with an arrow so that it forms the second-to-last arrow of a potential cycle cell, that move would allow your opponent to complete the cycle and win. We call such a move a \emph{death move} for that cell. We assume that players will want to avoid such moves.
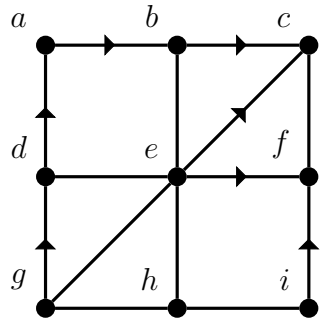
\begin{figure}[H]
	\centering
\begin{tikzpicture}[scale=1.75]
			 \begin{scope}[very thick, every node/.style={sloped,allow upside down}]
			 \Bvertex (A) at (-1,1) [label=above left:$a$]{};
			 \Bvertex (B) at (0,1) [label=above left:$b$]{};
			 \Bvertex (C) at (1,1) [label=above left :$c$]{};
			 \Bvertex (D) at (-1,0) [label=above left:$d$]{};
			 \Bvertex (E) at (0,0) [label=above left:$e$]{};
			 \Bvertex (F) at (1,0) [label=above left:$f$]{};
			 \Bvertex (G) at (-1,-1) [label=above left:$g$]{};
			 \Bvertex (H) at (0,-1) [label=above left:$h$]{};
			 \Bvertex (I) at (1,-1) [label=above left:$i$]{};
			 \draw (A) -- node {\midarrow} (B);
			 \draw (D) -- node {\midarrow} (A);
			 \draw (B) -- node {\midarrow} (C);
			 \draw (B) to (E);
			 \draw (E) -- node {\midarrow} (C);
			 \draw (C) to (F);
			 \draw (D) to (E);
			 \draw (G) -- node {\midarrow} (D);
			 \draw (E) -- node {\midarrow} (F);
			 \draw (H) to (E);
			 \draw (I) -- node {\midarrow} (F);
			 \draw (E) to (G);
			 \draw (G) to (H);
			 \draw (H) to (I);
			 \end{scope}
\end{tikzpicture}
	\caption{A board with some marked edges.}
	\label{gameplay}
\end{figure}

For instance, suppose it is your turn in the game board of Figure \ref{gameplay}. If you marked the edge $eg$ with an arrow from $e$ to $g$ (which we denote $e \to g$), that would be a death move for you, since your opponent could then play $d \to e$ to complete the cycle $e \to g \to d \to e$ and win.

Thus if you encounter a triangle with a single marked edge, avoiding death moves will restrict the direction that you can mark the other two edges of that triangle. If you now mark one of them with an arrow pointing opposite the death move direction, the triangle becomes \emph{uncyclable}, meaning that the triangle can no longer be made a cycle cell. In that case the remaining edge of the triangle is no longer restricted by avoiding a death move in this cell (though there may be a restriction imposed by an adjacent cell). This is an interesting aspect of the game:\ continued play can increase or decrease the number of potential future states of unmarked edges.

An edge where both possible markings are problematic (because the arrow would be a death move or create a sink/source), and where at least one is a death move for an adjacent cell, is said to be \emph{currently unplayable}.
Note that currently unplayable edges may become playable later.  In Figure \ref{gameplay}, the edge $de$ is currently unplayable, since marking it in either direction is a death move for you---your opponent could immediately complete the square cycle cell above it or a triangular cycle cell below it.

The sink-source rule also has some important consequences. For instance, the direction cannot change at a degree 2 vertex:\ if one edge is marked with an arrow pointing towards the vertex, the other edge may only be marked with an arrow pointing away, and vice versa. For example, in Figure \ref{gameplay}, vertex $i$ is a degree 2 vertex, so the edge $hi$ can only be marked with an arrow pointing towards $i$.

More generally, if a vertex has all but one edge marked with arrows pointing towards it and the remaining edge is unmarked, we call that vertex an \emph{almost-sink}.  The unmarked edge of an almost-sink can only be marked with an arrow pointing away from the vertex, or else it would violate the sink-source rule. We can also define an \emph{almost-source} in a similar fashion.  In Figure \ref{gameplay}, vertex $c$ is an almost-sink, and vertex $i$ is an almost-source.

We will call an unmarked edge \emph{markable} if it may be marked with an arrow without violating the sink-source rule.  When an unmarked edge is incident to two almost-sinks (respectively, to two almost-sources), we call such edges \emph{unmarkable} because neither marking of an edge is allowed:\ one direction forces a sink (respectively, source) at one neighboring vertex, and the other choice forces a sink (respectively, source) at the other vertex.    Once an edge becomes unmarkable, it stays unmarkable (in contrast to currently unplayable edges that can later become playable).  In Figure \ref{gameplay}, the edge $cf$ is unmarkable.


\section{Similar Games}

The Game of Cycles is an example of an impartial combinatorial game.  A \emph{combinatorial game} is a game in which two players alternate turns, with a clearly defined set of moves available and no hidden information or element of chance, until one of the players wins. A combinatorial game is called \emph{impartial} when the moves available in any configuration are the same for each player. 

%
%
%

Historically, graphs have been a fertile ground for playing and studying combinatorial games.  In the game of Dots and Boxes, first introduced in \cite{lucas1895arit}, a board has a grid of fixed vertices, and players take turns drawing edges between adjacent horizontal or vertical dots, attempting to completely enclose boxes when possible.  The game ends when all edges have been drawn, and the winner is the player who has enclosed the most boxes.  Another pair of games that can be played on more general graphs are Col and Snort \cite{conway2000numbers}. While they were originally presented as map-coloring games, they can equivalently be described as being played on planar graphs that represent the map, in which players take turns coloring vertices.  Depending on which game is being played, a player is either not allowed to color two adjacent vertices the same color, or not allowed to color two adjacent vertices different colors. In either case, once one player has no legal move available, the other player is declared the winner.  

Each of these games has a key distinction from the Game of Cycles we are studying.  Unlike Dots and Boxes, the Game of Cycles is played on a graph with both vertices and edges already present, with existing edges being labeled.  In addition, when a box is completed in Dots and Boxes, the game continues; once a cycle is formed in the Game of Cycles, the game is over.  The Game of Cycles is more similar to Col or Snort, though there are still significant differences.  Aside from the obvious distinction of marking edges instead of vertices, in Col or Snort each player is only allowed to mark a vertex with their specific color, whereas the Game of Cycles is impartial.  
However, the unmarkable edges resulting from the sink/source rule are similar to the unmarkable vertices often found in a completed game of Col or Snort.

In any combinatorial game, one can ask:\ who has a winning strategy? A \emph{strategy} is a specification of what moves to make in any situation, and since there are no draws, Zermelo's theorem~\cite{zermelo1913anwendung} tells us that one player has a \emph{winning strategy}, a way to force a win no matter how the other player plays. We shall be interested in determining a winning strategy for various classes of boards. 
Since the game is impartial, one might think about appealing to the Sprague-Grundy Theorem \cite{conway2000numbers}, but this is not a hopeful approach because the size of the game tree generally grows exponentially with the number of edges of the board. We shall instead appeal to the structure of the boards to develop winning strategies.







\section{Simple Boards}


In this section, we examine the game on some simple boards in order to develop some intuition into the strategy and mechanics of the game. We begin with playing on a $K_4$ board, the complete graph on four vertices embedded in a plane, and progress to a few other simple boards before proving general strategy theorems for boards with certain symmetries. 

Every planar realization of $K_4$ produces the same board, a collection of $2$-dimensional cells, edges, and vertices which looks like Figure \ref{K4}.
\begin{figure}[H]
			\begin{tikzpicture}[scale=1.75]
			 \begin{scope}[very thick, every node/.style={sloped,allow upside down}]
			 \Bvertex (A) at (0,0) [label=below:$a$]{};
			 \Bvertex (B) at (0,1)  [label=right:$b$] {};
			 \Bvertex (C) at (1,-0.7) [label=right:$c$] {};
			 \Bvertex (D) at (-1,-0.7) [label=left:$d$] {};
			 \draw (A) to (D);
			 \draw (D) to (C);
			 \draw (C) to (A);
			 \draw (B) to (C);
			 \draw (A) to (B);
			 \draw (D) to (B);
			 \end{scope}
			\end{tikzpicture}
	\caption{A $K_4$ board.}
	\label{K4}
\end{figure}
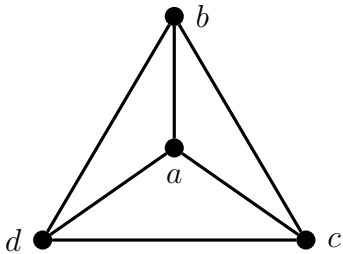

\begin{thm}
\label{k4board}
On a $K_4$ board, Player~2 has a winning strategy. 
\end{thm}


\begin{proof}
We describe a winning strategy for Player 2. 
After Player~1 makes the first move, Player~2 should respond by marking the unique edge that is not incident to the the first marked edge. This produces a board with exactly one external edge and one internal edge marked, which, up to rotation or reflection, must look like either Board 1 or Board 2:

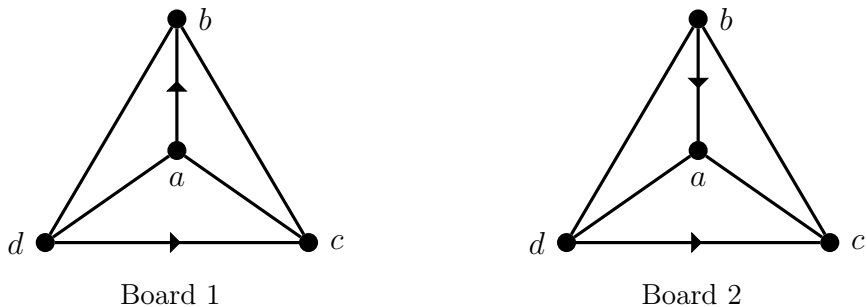
\begin{figure}[H]
	\captionsetup[subfigure]{labelformat=empty}
	\centering
	\subcaptionbox{Board 1}{
			\begin{tikzpicture}[scale=1.75]
			 \begin{scope}[very thick, every node/.style={sloped,allow upside down}]
			 \Bvertex (A) at (0,0) [label=below:$a$]{};
			 \Bvertex (B) at (0,1)  [label=right:$b$] {};
			 \Bvertex (C) at (1,-0.7) [label=right:$c$] {};
			 \Bvertex (D) at (-1,-0.7) [label=left:$d$] {};
			 \draw (A) to (D);
			 \draw (D)-- node {\midarrow} (C);
			 \draw (C) to (A);
			 \draw (B) to (C);
			 \draw (A) -- node {\midarrow} (B);
			 \draw (D) to (B);
			 \end{scope}
			\end{tikzpicture}
	}
	\hskip.7in
	\subcaptionbox{Board 2}{
			\begin{tikzpicture}[scale=1.75]
			\begin{scope}[very thick, every node/.style={sloped,allow upside down}]
			\Bvertex (A) at (0,0) [label=below:$a$]{};
			\Bvertex (B) at (0,1)  [label=right:$b$] {};
			\Bvertex (C) at (1,-0.7) [label=right:$c$] {};
			\Bvertex (D) at (-1,-0.7) [label=left:$d$] {};
			\draw (A) to (D);
			\draw (D) -- node {\midarrow} (C);
			\draw (C) to (A);
			\draw (B) to (C);
			\draw (B) -- node {\midarrow} (A);
			\draw (D) to (B);
			\end{scope}
			\end{tikzpicture}
	}
	\caption{Possible boards
	resulting from the first moves of Players~1 and 2.}
	\label{K4_1}
\end{figure}

Consider first Board 1, the case where the edge $ab$ is marked with an arrow $a \to b$. Note that the edge $ad$ is currently unplayable:\ marking it $a \to d$ is a death move for Player~1 because the bottom cell can be immediately cycled by Player~2, and marking it $d \to a$ is a death move Player~1 because the left cell can immediately be cycled by Player 2. Thus there are three available edges on which Player~1 can play on their second move. Note that on each of these edges, there is only one move that is not a death move, since all three cells already have a cycle direction chosen by the first two moves. After Player~1 makes their move, Player~2 can ensure that the board becomes one of the following two:

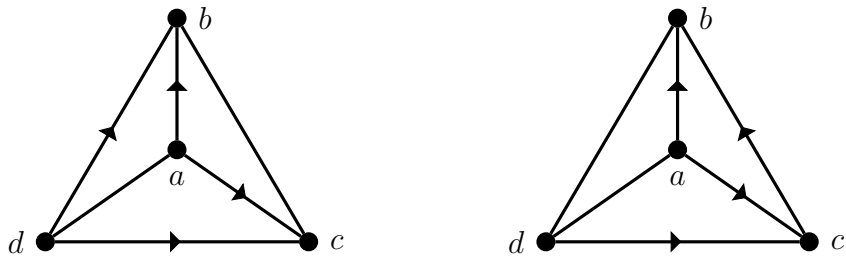
\begin{figure}[H]
	\centering
		\begin{tikzpicture}[scale=1.75]
		\begin{scope}[very thick, every node/.style={sloped,allow upside down}]
		\Bvertex (A) at (0,0) [label=below:$a$]{};
		\Bvertex (B) at (0,1)  [label=right:$b$] {};
		\Bvertex (C) at (1,-0.7) [label=right:$c$] {};
		\Bvertex (D) at (-1,-0.7) [label=left:$d$] {};
		\draw (A) to (D);
		\draw (D)-- node {\midarrow} (C);
		\draw (A) -- node {\midarrow} (C);
		\draw (B) to (C);
		\draw (A) -- node {\midarrow} (B);
		\draw (D) -- node {\midarrow}  (B);
		\end{scope}
		\end{tikzpicture}
	\hskip.7in
		\begin{tikzpicture}[scale=1.75]
		\begin{scope}[very thick, every node/.style={sloped,allow upside down}]
		\Bvertex (A) at (0,0) [label=below:$a$]{};
		\Bvertex (B) at (0,1)  [label=right:$b$] {};
		\Bvertex (C) at (1,-0.7) [label=right:$c$] {};
		\Bvertex (D) at (-1,-0.7) [label=left:$d$] {};
		\draw (A) to (D);
		\draw (D) -- node {\midarrow} (C);
		\draw (A) -- node {\midarrow} (C);
		\draw (C) -- node {\midarrow}  (B);
		\draw (A) -- node {\midarrow} (B);
		\draw (D) to (B);
		\end{scope}
		\end{tikzpicture}
	\caption{Possible boards resulting from the second non-death moves  of  Players~1 and 2 on Board 1 in Figure~\ref{K4_1}.}
	\label{K4_2}
\end{figure}
In the first case, the game is over because the remaining edges are unmarkable due to the sink-source rule. In the second case, only death moves for Player~1 remain. Therefore Player~2 wins in both cases.

The argument for Board 2, in which edge $ab$ is marked $b \to a$, is analogous, leading to Player 2 winning. Thus Player 2 has a winning strategy. 
\end{proof}

Now let's consider a $C_n$ board:\ a single cell whose boundary is a \emph{cycle graph} $C_n$ with $n$ vertices and $n$ edges alternating along the boundary of the cell.

 \begin{figure}[H]
	\centering
\begin{tikzpicture}[scale=1]
			 \begin{scope}[very thick, every node/.style={sloped,allow upside down}]
			 \Bvertex (A) at (2,0){};
			 \Bvertex (B) at (1.414, 1.414){};
			 \Bvertex (C) at (0,2){};
			 \Bvertex (D) at (-1.414,1.414){};
			 \Bvertex (E) at (-2,0){};
			 \Bvertex (F) at (-1.414,-1.414){};
			 \Bvertex (G) at (0,-2){};
			 \Bvertex (H) at (1.414,-1.414){};
			 \draw (A) -- node {\midarrow} (B);
			 \draw (B) to (C);
			 \draw (D) -- node {\midarrow} (C); 
			 \draw (E) -- node {\midarrow} (D);
			 \draw (E) to (F);
			 \draw (F) to (G); 
			 \draw (G) to (H);
			 \draw (H) to (A);
			 \end{scope}
\end{tikzpicture}
	\caption{A cycle board with 8 vertices and a few marked edges.}
	\label{cyclegraph}
\end{figure}
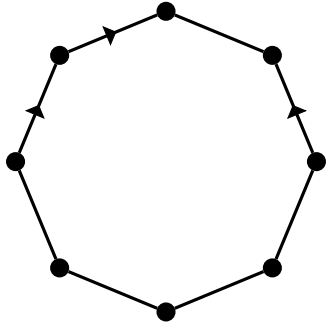

The following lemma will be useful to analyze a $C_n$ board.
\begin{lem}
\label{paritylem}
If a $C_n$ board has no markable edges, the number of unmarkable edges must be even.
\end{lem}

Note that to say a board has no markable edges means that the game has played out as far as it can play--no further moves are possible.

\begin{proof}
We first show if every there are no markable edges, then there cannot be two unmarked edges adjacent to one another. For if there were adjacent unmarked edges $ab$ and $bc$, then vertex $b$ would not be an almost-sink or almost-source, hence $bc$ (for instance) could not be unmarkable. Thus $bc$ is markable, a contradiction.

Next notice that two adjacent marked edges must be pointing in the same direction (clockwise or counterclockwise), to keep the vertex in between from being a sink or a source.

Thus the unmarkable edges must divide the boundary of the $C_n$ board into uni-directional chains of marked edges, and the direction must change at each unmarked edge. Since starting at some edge and moving around the circle brings us back to the same edge, the number of direction changes must be even.
\end{proof}

\begin{thm}
\label{paritythm}
The play on a $C_n$ board is entirely determined by parity.
If $n$ is odd, Player~1 wins. If $n$ is even, Player~2 wins.
\end{thm}

Note that saying a player `wins' is stronger than saying a player has a winning strategy---gameplay is deterministic and the outcome does not depend on the sequence of moves made.

\begin{proof}
Since there is only one cell, and a cycle cell around it happens only if there is no unmarkable edge, the winner will be determined solely by the parity of the number of moves made. The number of moves made is equal to the number of edges minus the number of unmarkable edges in the final board state. The result follows since the number of unmarkable edges must be even.  So we see that as long as players make legal moves, there is no real ``strategy'' needed at all.
\end{proof}




\begin{thm}
\label{ngonchordthm}
Let $n \geq 4$.
Consider a $C_n$ board subdivided into two cells by one internal chord connecting two non-adjacent vertices. 
If $n$ is even, then Player~1 has a winning strategy, and if $n$ is odd, then Player~2 has a winning strategy. 
\end{thm}

 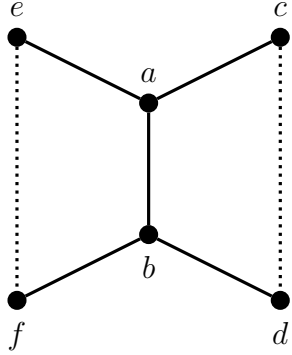
\begin{figure}[H]
	\centering
\begin{tikzpicture}[scale=1.75]
			 \begin{scope}[very thick, every node/.style={sloped,allow upside down}]
			 \Bvertex (A) at (0,0.5) [label=above:$a$] {};
			 \Bvertex (B) at (0,-0.5) [label=below:$b$] {};
			 \Bvertex (C) at (1,1)  [label=above:$c$] {};
			 \Bvertex (D) at (1,-1)  [label=below:$d$] {};
			 \Bvertex (E) at (-1,1)  [label=above:$e$] {};
			 \Bvertex (F) at (-1,-1)  [label=below:$f$] {};
			 \draw (A) to (B);
			 \draw (A) to (C); 
			 \draw (B) to (D); 
			 \draw (A) to (E);
			 \draw (B) to (F); 
			 \draw[dotted] (C) to (D);
			 \draw[dotted] (E) to (F);
			 \end{scope}
\end{tikzpicture}
	\caption{A cycle board with an internal chord.}
	\label{ngonchord}
\end{figure}
\begin{proof}
Call the player who has a winning strategy Player~$W$.
To prove the theorem, we will show that,
unless $W$ wins because the other player plays a death move prematurely, Player~$W$ can force unmarkable edges to come in pairs on this graph by the time the game ends. 
Then the result follows from the parity of the number of edges marked, as in the proof of Theorem \ref{paritythm} above, and because if $n$ is even, the total number of edges is odd, and if $n$ is odd then the total number of edges is even.

The region immediately surrounding the chord looks like Figure \ref{ngonchord}
where the dotted lines represent identification or a collection of edges joined by vertices of degree 2.
The graph has exactly two degree 3 vertices $a$ and $b$, the endpoints of this chord.

If $n$ is even, we claim Player~$W$ is Player~1, and their first move should be to mark the chord $ab$ with an arrow in either direction. The direction of this arrow will suggest an orientation of the boundary of each of the two cells that border it,
one clockwise and the other counterclockwise. For instance, suppose Player 1 marks the chord 
with an arrow $a \to b$.
Player~1's next two plays should guarantee that one edge incident to $a$ and one edge incident to $b$ are each marked in the direction opposite to that implied by the chord (if the other player doesn't make those moves). This means placing an arrow either $a \to c$ or $a \to e$, 
and placing an arrow either $d \to b$ or $f \to b$.
Such moves are always possible at each degree 3 vertex, since even if Player~2 plays `with' the chord on one of the edges incident to that vertex, Player~1 can play as desired on the other edge incident to that vertex on their next turn.

If $n$ is odd, we claim Player~$W$ is Player~2, and a winning strategy relies on the same key idea of marking edges in direction opposite that implied by the chord.  If Player~1 marks the chord on their first turn, Player~2 should immediately play on one of the incident edges in the direction opposite to that implied by the chord.  If Player~1 instead plays initially on an edge incident to the chord (for example $c \to a$) 
or incident to one of those edges (for example, the other edge incident to $c$, say $x \to c$), Player~2 should mark the chord in the direction opposite that implied by that edge (in this case, $b \to a$).  After the next move by Player~1, or following any other initial Player~1 move, Player~2 can now follow the same strategy that Player~1 employed above to ensure that there are edges incident to both $a$ and $b$ with arrows marked in directions opposite to that implied by the chord.

At this point, whether $n$ is odd or even, either the left cell or the right cell has a direction change at $a$ (but not both, else $a$ would be a sink or source). Similarly, either (but not both) the left or right cell has a direction-change at $b$.  A key observation is that when the game concludes, if there isn't a completed cycle, all the direction changes in a cell must occur at ummarkable edges or possibly at $a$ or $b$, because there cannot be a direction change at any other vertex (else it would be a sink or source, being of degree 2). As in the proof of Lemma~\ref{paritylem}, the total number of direction changes within a cell must be even.

So after initially ensuring a direction change at $a$ and at $b$, Player~$W$'s winning strategy is to complete a cycle if the other player plays a death move, and otherwise Player~$W$ should just avoid making a death move until the game concludes with no markable edges. Such a move will always be possible because if the only option for Player~$W$ in one cell is to play a death move, then all but 2 edges in that cell are marked and the other cell must have direction changes at both $a$ and $b$. If this is Player~$W$'s move, a comparison of the parity of the total number of edges and the number of moves played will show that there must be an odd number of unmarked edges in the other cell. These cannot all be unmarkable (else the number of direction changes in that cell would be odd).

Thus we may as well assume that the game ends when there are no markable edges.  At that time, there are exactly two direction changes at $a$ and $b$, and there must be an even number of direction changes in each cell. Together, these observations imply that any unmarkable edges come in pairs, as we set out to show.

\end{proof}

We point out that unlike Theorem~\ref{paritythm}, the 
game play on the board of Theorem~\ref{ngonchordthm} is not deterministic.  
In non-optimal play, if there is just a single direction change at the degree 3 vertices, as in Figure~\ref{ngonchordloss}, then the same argument illustrates that there will be an odd number of unmarkable edges in a finished game, so the player who has a winning strategy will not win in this case.

 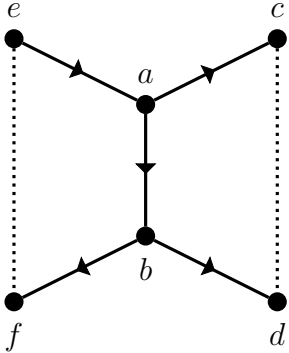
\begin{figure}[H]
	\centering

\begin{tikzpicture}[scale=1.75]
			 \begin{scope}[very thick, every node/.style={sloped,allow upside down}]
			 \Bvertex (A) at (0,0.5) [label=above:$a$] {};
			 \Bvertex (B) at (0,-0.5) [label=below:$b$] {};
			 \Bvertex (C) at (1,1)  [label=above:$c$] {};
			 \Bvertex (D) at (1,-1)  [label=below:$d$] {};
			 \Bvertex (E) at (-1,1)  [label=above:$e$] {};
			 \Bvertex (F) at (-1,-1)  [label=below:$f$] {};
			 \draw (A)-- node {\midarrow} (B);
			 \draw (A) -- node {\midarrow} (C); 
			 \draw (B)-- node {\midarrow} (D); 
			 \draw (E) -- node {\midarrow} (A);
			 \draw (B)-- node {\midarrow} (F); 
			 \draw[dotted] (C) to (D);
			 \draw[dotted] (E) to (F);
			 \end{scope}
\end{tikzpicture}
	\caption{A board played without using the winning strategy.}
	\label{ngonchordloss}
\end{figure}

The game play on the board of Theorem~\ref{ngonwithflap} below is also not deterministic.

\begin{thm}
\label{ngonwithflap}
Let $n \geq 3$. Consider a board with two cells, formed by a cycle graph $C_n$ with one additional internal vertex $a$ that has edges to exactly two adjacent vertices $b$ and $c$ on the cycle $C_n$.
If $n$ is odd, then Player~1 has a winning strategy and if $n$ is even, then Player~2 has a winning strategy.
\end{thm}

 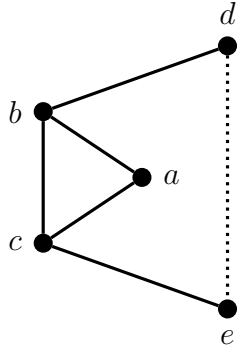
\begin{figure}[H]
	\centering
\begin{tikzpicture}[scale=1.75]
			 \begin{scope}[very thick, every node/.style={sloped,allow upside down}]
			 \Bvertex (A) at (-0.25,0) [label=right:$a$]{};
			 \Bvertex (B) at (-1,0.5) [label=left:$b$] {};
			 \Bvertex (C) at (-1,-0.5) [label=left:$c$] {};
			 \Bvertex (D) at (0.4,1)  [label=above:$d$] {};
			 \Bvertex (E) at (0.4,-1)  [label=below:$e$] {};
			 \draw (A) to (B);
			 \draw (A) to (C);
			 \draw (B) to (C); 
			 \draw (B) to (D);
			 \draw (C) to (E); 
			 \draw[dotted] (D) to (E);
			 \end{scope}
\end{tikzpicture}
	\caption{A cycle board with an additional internal vertex adjoining adjacent vertices on the cycle.}
	\label{Cflap1}
\end{figure}


\begin{proof}
The outcome of the game will be determined by what happens in the vicinity of the degree 3 vertices $b$ and $c$ (see Figure~\ref{Cflap1}).
Vertices $d$ and $e$ are either identified (in the case $n=3$) or
are connected by a chain of edges and degree 2 vertices (in the case $n>3$), which we've represented with a dotted line. Note that this board contains two cells:\ the triangular cell $abc$ and a larger cell with $n+1$ edges. 

The winning strategy here follows a pretty simple idea:\
don't make a death move. As before, let Player~$W$ denote the player who has a winning strategy. 
We'll show, as in Theorem \ref{ngonchordthm}, that Player~$W$ can avoid making a death move until the game is in a state where the winner is determined by parity. 
It will be sufficient to show that Player~$W$ can ensure that the triangular cell is uncyclable, and from that point on, finishing the game is the same as concluding the game on the larger cell alone.


If Player~$W$ is Player~1, then Player~$1$ should play $b \to c$ on their first move. If Player~$2$ makes a death move in the triangular cell, then Player~1 can complete the cycle and win.  Otherwise, if Player~2 initially plays anything else on the triangular cell (suppose it is $a \to c$), that move will make the triangular cell uncyclable, 
and else if Player~2 initially marks an edge on the outer boundary of the larger cell, 
Player~1 can always mark either $a \to c$ or $b \to a$ to make the triangular cell uncyclable. From this point, the game reduces to the game on the larger cycle, because now a move is valid on the original board if and only if it is valid on the board with only the larger cell.


Now suppose Player~$W$ is Player~2. Then if Player~1 plays the chord $b \to c$ on their first move (an argument for $c \to b$ is similar), Player~2 should respond by marking $a \to c$ to make the triangular cell uncyclable.  And if Player~1 initially plays another edge on the triangular cell (suppose it is $a \to c$), then Player~2 should play so as to make the triangular cell uncyclable (here, $bc$). Finally, if Player~1 initially plays any other edge (for example, $c \to e$), Player~2 should mark the chord $b \to c$ in a direction opposite that implied by Player~1's move (in this case, $c \to b$).  Doing so will ensure that a subsequent move is valid on the original board if and only if it is valid on the board with only the larger cell. Player~2's second move should then be chosen to make the triangular cell uncyclable (unless Player~1's second move is a death move).


At this point in either scenario, play proceeds as if it were only on the larger cell, a cycle of size $n+1$.  The result follows by noting that play ends when there are no unmarkable edges, they must come in pairs, and the parity of the number of edges played determines the winner. 
\end{proof}




\section{Symmetric Boards}

We now consider boards with various kinds of symmetry.


For example, suppose we consider a board with $180^{\circ}$ rotational symmetry. The $180^{\circ}$ rotational symmetry of the board ensures that every vertex $v$ has a \emph{partner vertex} $v'$ that it is paired with under a $180^{\circ}$ rotation. The only case where $v$ is identical to $v'$ is if $v$ is a vertex at the center of the rotation.  Similarly, every edge has a \emph{partner edge} that it corresponds to under this symmetry of this board. 

\begin{figure}[H]
    \centering
\begin{tikzpicture}[scale=1.25]
		\begin{scope}[very thick, every node/.style={sloped,allow upside down}]
		\Bvertex (A) at (-1,1) [label=left:$a$]{};
		\Bvertex (B) at (1,1)  [label=right:$b$] {};
		\Bvertex (C) at (1,-1) [label=right:$c$] {};
		\Bvertex (D) at (-1,-1) [label=left:$d$] {};
		\Bvertex (E) at (0,0) [label=above:$e$] {};
		\draw (A) to (B);
		\draw (B) to (C);
		\draw (C) to (D);
		\draw (D) to (A);
		\draw (A) to (C);
		\draw (D) to (B);
		\end{scope}		
\end{tikzpicture}
    \caption{A board with $180^\circ$ rotational symmetry.}
    \label{fig:rot}
\end{figure}
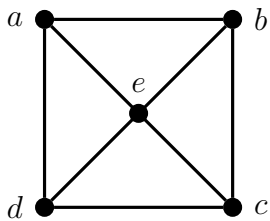

For example, looking at Figure \ref{fig:rot} above, we see that vertex $a$ and vertex $c$ are partners, vertex $b$ vertex $d$ are partners, edges $ab$ and $cd$ are partners, and edges $de$ and $be$ are partners.  This suggests that Player~2 might have a ``mirror-reverse'' strategy to win any game on this board:
\begin{enumerate}
    \item If possible to win by completing a cycle, do so.
    \item If that is not possible, mirror Player~1's strategy by observing Player~1's most recent move $i \to j$ and playing $j' \to i'$, the partner edge of $ij$ with its arrow reversed.
\end{enumerate}
For example, in Figure \ref{fig:rot}, if Player~1's first move is to mark the edge $a \to b$, Player~2 would respond by marking the edge $d \to c$. Here is one possible complete game where Player~2 follows this mirror-reverse strategy with rotation:
\begin{itemize}
    \item Player~1 marks $a \to b$.
    \item Player~2 marks $d \to c$.
    \item Player~1 marks $a \to e$.
    \item Player~2 marks $e \to c$.
We note at this point that $a$ is an almost-source, and $c$ is an almost-sink.  
    \item Player~1 marks $e \to b$.
    \item Player~2 marks $d \to e$.
\end{itemize}
At this point the remaining two edges are unmarkable, and therefore Player~2 wins.

We should worry whether Player~2 is always able to make a move---in other words, the prescribed edge to be marked is available, does not create a source/sink, and does not produce a death move for Player~2.  

Indeed, this mirroring strategy may require Player 2 to make a death move in some instances. Consider Figure \ref{fig:rot-problem} below and possible moves made by Player~1 and Player~2 in a game.

\begin{figure}[H]
    \centering
\begin{tikzpicture}[scale=1.25]
		\begin{scope}[very thick, every node/.style={sloped,allow upside down}]
		\Bvertex (A) at (-1,1) [label=left:$a$]{};
		\Bvertex (B) at (1,1)  [label=right:$b$] {};
		\Bvertex (C) at (1,-1) [label=right:$c$] {};
		\Bvertex (D) at (-1,-1) [label=left:$d$] {};
		\Bvertex (E) at (0,0) [label=above:$e$] {};
		\draw (A) to (B);
		\draw (B) to (C);
		\draw (C) to (D);
		\draw (D) to (A);
		\draw (E) to (A);
		\draw (E) to (C);
		\end{scope}		
\end{tikzpicture}
    \caption{A board with $180^\circ$ rotational symmetry.}
    \label{fig:rot-problem}
\end{figure}

\begin{itemize}
    \item Player~1 marks $b \to a$.
    \item Player~2 marks $c \to d$.
    \item Player~1 marks $a \to e$.
    \item Player~2 marks $e \to c$.
\end{itemize}
Oops, we are in trouble, since that was a death move. 
Player~1 can close either cycle and win. The problem appears to be that Player~2's response to Player~1's added to a cycle that Player 1's move was contributing to. 
We could prevent this situation from occurring if we require that no edge and its partner border the same cell.

Now what if we had a board like Figure \ref{fig:rot-edge}, where one edge is its own partner under $180^{\circ}$ rotational symmetry?

\begin{figure}[H]
    \centering
\begin{tikzpicture}[scale=1.25]
		\begin{scope}[very thick, every node/.style={sloped,allow upside down}]
		\Bvertex (A) at (-1,1) [label=left:$a$]{};
		\Bvertex (B) at (1,1)  [label=right:$b$] {};
		\Bvertex (C) at (1,-1) [label=right:$c$] {};
		\Bvertex (D) at (-1,-1) [label=left:$d$] {};
		\Bvertex (E) at (-0.3,0) [label=left:$e$] {};
		\Bvertex (F) at (0.3,0) [label=right:$f$] {};
		\draw (A) to (B);
		\draw (B) to (C);
		\draw (C) to (D);
		\draw (D) to (A);
		\draw (A) to (E);
		\draw (D) to (E);
		\draw (E) to (F);
		\draw (F) to (B);
		\draw (F) to (C);
		\end{scope}		
\end{tikzpicture}
    \caption{A board with $180^\circ$ rotational symmetry.}
    \label{fig:rot-edge}
\end{figure}

In this case, notice that if Player~1 plays the self-symmetric edge, marking $e \to f$ first, then the remainder of the board can be played just as before, with Player~2 now the ``first player'', and Player~1 always mirroring Player~2's moves. 
This suggests a mirror-reversing strategy for Player 1 could be a winning strategy, provided no edge and its partner are part of the same cell.

For example, a complete game could look like:
\begin{itemize}
    \item Player~1 marks $e \to f$.
    \item Player~2 marks $e \to d$.
    \item Player~1 marks $b \to f$. At this point edges $ea$ and $fc$ are currently unplayable.
    \item Player~2 marks $a \to e$. A mistake by Player~2.
    \item Player~1 marks $d \to a$, and wins.
\end{itemize}

We shall prove that the mirror-reversing strategy works for boards with rotational symmetry, but also for boards with a reflective symmetry, as in Figure \ref{fig:ref}.

\begin{figure}[H]
    \centering  
    \begin{tikzpicture}[scale=1.25]
		\begin{scope}[very thick, every node/.style={sloped,allow upside down}]
		\Bvertex (A) at (0,2)  [label=right:$a$] {};
		\Bvertex (B) at (1,1)  [label=right:$b$] {};
		\Bvertex (C) at (1,-1) [label=right:$c$] {};
		\Bvertex (D) at (-1,-1) [label=left:$d$] {};
		\Bvertex (E) at (-1,1) [label=left:$e$]{};
		\Bvertex (F) at (0,0) [label=above:$f$] {};

		\draw (A) to (B);
		\draw (A) to (E);
		\draw (B) to (C);
		\draw (D) to (E);
		\draw (E) to (C);
		\draw (D) to (B);
		
		\end{scope}		
    \end{tikzpicture}
    \caption{A board with reflective symmetry.}
    \label{fig:ref}
\end{figure}
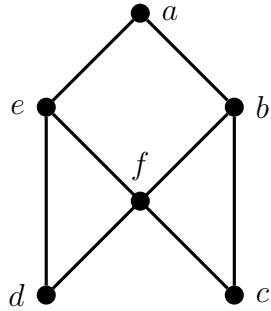

This board also has natural pairings of edges into partners, as before, but this time we use reflection to produce partners.  For example, looking at Figure \ref{fig:ref} above, we see that vertices $b$ and $e$ are partners, and vertices $c$ and $d$ are partners.  Vertex $a$ is a partner to itself, as is vertex $f$.   Similarly, edge $ab$ is partner to edge $ae$, and edge $de$ is partner to edge $cb$. 

For this board, here is a possible game where Player~2 follows a mirror-reverse strategy with reflection.
\begin{itemize}
    \item Player~1 marks $a \to b$.
    \item Player~2 marks $e \to a$.
    \item Player~1 marks $b \to c$.
    \item Player~2 marks $d \to e$.
At this point, all legal moves remaining for Player~1 are death moves. We choose one for Player~1 to make below.
    \item Player~1 marks $e \to f$.
    \item Player~2 marks $f \to d$, and completes a cycle to win.
\end{itemize}

Now consider the board in Figure~\ref{fig:ref-problem} below.

\begin{figure}[H]
    \centering
    \begin{tikzpicture}[scale=1.25]
		\begin{scope}[very thick, every node/.style={sloped,allow upside down}]
		\Bvertex (A) at (0,2)  [label=right:$a$] {};
		\Bvertex (B) at (1,1)  [label=right:$b$] {};
		\Bvertex (C) at (1,-1) [label=right:$c$] {};
		\Bvertex (D) at (-1,-1) [label=left:$d$] {};
		\Bvertex (E) at (-1,1) [label=left:$e$]{};
		\Bvertex (F) at (0,0) [label=above:$f$] {};

		\draw (A) to (B);
		\draw (A) to (E);
		\draw (B) to (C);
		\draw (D) to (E);
		\draw (E) to (C);
		\draw (D) to (B);
		
		\draw (E) to (B);
		\end{scope}		
    \end{tikzpicture}
    \caption{A board with a fixed edge under reflective symmetry.}
    \label{fig:ref-problem}
\end{figure}
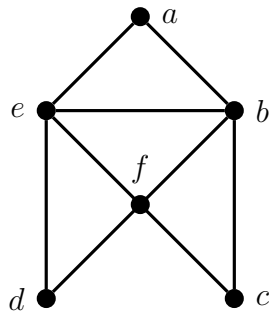

Note that the edge $eb$ is a fixed edge under the reflection.  In this case Player 1 has a winning strategy, which would be to first play the fixed edge (in either direction), and thereafter, no matter what Player~2 does, Player~1 can mirror-reverse Player~2's moves (as long as completing a cycle isn't an option).

Here is a possible game where Player~2 follows this strategy in Figure~\ref{fig:ref-problem}.
\begin{itemize}
    \item Player~1 marks $b \to e$.
    \item Player~2 marks $e \to a$.
    \item Player~1 marks $a \to b$, and wins.
\end{itemize}

This strategy, however may fail for the simple board in Figure~\ref{fig:ref-problem1} with reflective symmetry.

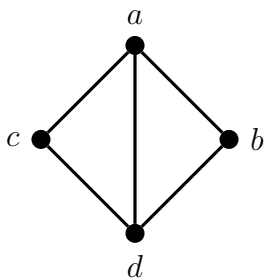
\begin{figure}[H]
    \centering
    \begin{tikzpicture}[scale=1.25]
		\begin{scope}[very thick, every node/.style={sloped,allow upside down}]
		\Bvertex (A) at (0,2)  [label=above:$a$] {};
		\Bvertex (B) at (1,1)  [label=right:$b$] {};
		\Bvertex (C) at (-1,1) [label=left:$c$]{};
		\Bvertex (D) at (0,0) [label=below:$d$] {};

		\draw (A) to (B);
		\draw (A) to (C);
		\draw (C) to (D);
		\draw (B) to (D);
		\draw (A) to (D);
		\end{scope}		
    \end{tikzpicture}
    \caption{A board with an edge along the axis of reflective symmetry.}
    \label{fig:ref-problem1}
\end{figure}
Here is a possible game where Player~1 follows this strategy.
\begin{itemize}
    \item Player~1 marks $a \to d$.
    \item Player~2 marks $b \to d$.
    \item Player~1 marks $d \to c$ following the strategy, but this is a death move.
    \item Player~2 marks $c \to a$, and wins.
\end{itemize}
In this case the problem arose because the marked edge $ad$ is in both cells, but it is not its own mirror-reverse. So it contributes towards building a cycle in one cell but not the other.  So Player~1 does not win this board if they follow a mirror-reversing strategy employing reflection, but notice that Player~1 could win using a mirror-reversing strategy employing $180^{\circ}$ rotation.

We now state our theorem that covers both rotational and reflective symmetry, because these symmetries are examples of \emph{involutive symmetry}.  A board has involutive symmetry if there is a non-trivial symmetry $\tau$ of the board which is its own inverse.  In this case, $\tau$ is called an \emph{involution}.  Any involution assigns a unique partner to each vertex, edge, and cell.  

Call a vertex, edge, or cell \emph{self-involutive} if the involution of that vertex, edge, or cell is itself.  Note that a self-involutive edge may not necessarily fix vertices, and a self-involutive cell may not necessarily fix edges or vertices.
Call a cell \emph{part-involutive} if the cell is not self-involutive, but at least one edge in the cell has its partner also in the cell.
Call a cell \emph{nowhere-involutive} if no edge of the cell has its partner in the cell.  Thus, a cell is either self-involutive, part-involutive, or nowhere-involutive.

\begin{thm}\label{invsym}
Let $G$ be a board with an involution such that each cell is either self-involutive or nowhere-involutive.  If there is no self-involutive edge, then Player~2 has a winning strategy. If there is exactly one self-involutive edge whose vertices are not fixed by the involution, then Player~1 has a winning strategy.
\end{thm}



\begin{proof}
If $i$ is a vertex, call its involution partner $i'$.  Similarly, for an edge $e$ or cell $\sigma$, let $e'$ and $\sigma'$ denote their partners under the involution.  If $ij$ denotes the edge with endpoints $i$ and $j$, we say `playing $i \to j$' means a player will mark an arrow from $i$ to $j$. 

We claim that the player with the winning strategy, whom we'll call Player~$W$, can win by using a ``mirror-reverse'' strategy in responding to the other player, whom we'll call Player~$X$. Player~$W$'s mirror-reverse strategy is this:\ Player~$W$ should complete a cycle if possible, but otherwise Player~$W$ should look at the arrow Player~$X$ played on the previous move (say, $i \to j$) and play the mirror-reverse (by playing $j' \to i'$) on the next move.

If there is no self-involutive edge, then we'll show Player~2 has the mirror-reverse strategy as a winning strategy (with Player~2 in the role of Player~$W$ and Player~1 in the role of Player~$X$).

If there is a self-involutive edge whose vertices are not fixed by the involution, then we show that Player~1 has a winning strategy by playing the self-involutive edge (in either direction) as a first move, then using the mirror-reverse strategy in responding to subsequent moves by Player~2 (with Player~1 in the role of Player~$W$ and Player~2 in the role of Player~$X$). Note that, since we assumed the involution doesn't fix vertices of the self-involutive edge $mn$, the vertices $m$ and $n$ must be partners. This implies that an arrow $m \to n$ is itself fixed under the mirror-reverse operation, since the arrow is first reversed by the mirroring involution and then reversed again by the reversal step. (This property is essential to ensure that the arguments below about playing partner edges will hold without trouble.)

We'll now check that this strategy for Player~$W$ is always available, will never lead to a sink or source for Player~$W$, and will not lead to a death move for Player~$W$. This implies Player~$X$ cannot complete a cycle or make the final move, so Player~$W$ wins.

Let's assume Player~$W$ has followed this strategy for the entire game so far, and that Player~$X$ plays $i \to j$ in the current move.  We first note that the move $j' \to i'$ must be available to $X$, because (1) if that partner edge had been previously played by Player~$X$, then Player~$W$ would have played $i \to j$ before the current move, and (2) if that partner edge had been previously played by Player~$W$, it would have been the response to Player~$X$ playing $i \to j$ before the current move.

Now we confirm that the mirroring move will always be legal for Player~$W$, that is, it will not produce a sink or source at some vertex $v$. For suppose it did, and suppose $v \neq v'$. Since every edge incident to $v$ has now been marked, every edge incident to $v'$ has also been marked. Since the mirroring strategy involves switching directions when marking the partner edge, if $v$ is a source (or sink) than $v'$ is a sink (or source). Therefore the move by Player~$X$ that Player~$W$ is mirroring must already have violated the sink-source rule, a contradiction. In the case $v=v'$ (which happens when $v$ is self-involutive) we see that $v$ cannot be a source or sink after Player~$W$'s mirroring move, since Player~$X$'s and Player~$W$'s moves together result in one arrow pointing towards $v$ and another pointing away. For instance, in the board of Figure \ref{fig:rot} above, if Player~$X$ marks $d \to e$, Player~$W$ would respond by marking $e \to b$.  

Next, we note that this move for Player~$W$ is not a death move, i.e., this move will never permit Player~$X$ to play an edge that will complete a cycle.
For if, after $W$'s move, there were an edge $e$ that Player~$X$ could mark to complete a cycle cell around a cell $\sigma$, then $W$'s move must have produced an almost-cycle:\ all but one edge in that cycle must already be marked with arrows cycling in one direction around the cell $\sigma$. 
By assumption, $\sigma$ is either self-involutive or nowhere-involutive. 
In the case that $\sigma$ is self-involutive, the final unmarked edge $e$ must be a self-involutive edge (else if $e$'s partner edge in $\sigma$ were distinct, then it would have been played earlier, which would imply that $e$ had already been played before this move).  But if $e$ were self-involutive, it would have been played as the first move of the game (by Player~1) so it cannot be available now.
And in the case that $\sigma$ is nowhere-involutive, consider the almost-cycle in $\sigma$ and look at the corresponding partner edges---these must have formed an almost-cycle in $\sigma'$ prior to Player~$W$'s move in $\sigma$. But then the mirror-reversing strategy would have urged Player~$W$ to mark the partner edge $e'$ to complete the almost-cycle in $\sigma'$ rather than play the current move to create an almost-cycle in $\sigma$.
Together, these observations show that this move for Player~$W$ cannot have been a death move.
\end{proof}

With the theorem we obtain some immediate corollaries as special cases.

\begin{cor}\label{rotsym}
Let $G$ be a board with $180^{\circ}$ rotational symmetry, and no edge and its partner part of the same cell. If there is no edge through the center of the board then Player~2 has a winning strategy. If there is such an edge, then Player~1 has a winning strategy.
\end{cor}

This follows from Theorem \ref{invsym} by noting that if no edge and partner is of the same cell, then every cell is nowhere-involutive.

\begin{cor}\label{mirrorsymbasic}
Let $G$ be a board that is symmetric by reflection across some line, with no edges along that axis of symmetry and at most one edge crossing that axis of symmetry.
On this board, Player~2 has a winning strategy if there is no edge crossing this axis of symmetry. If there is a single edge crossing this axis of symmetry, Player~1 has a winning strategy.
\end{cor}

This can be proved by noting the hypotheses guarantee that there is at most a single self-involutive edge, and no cell is part-involutive, so Theorem \ref{invsym} applies. Note that this theorem does not say anything about boards with multiple self-involutive edges.

\section{Filled Board Theorem}

Because of the sink-source rule, the Game of Cycles has a topological flavor that is reminiscent of the Brouwer fixed point theorem or theorems about nonvanishing vector fields on balls.  In particular, the arrows along edges are like a `discrete vector field' and the source-sink rule prevents the existence of a `vanishing point' at vertices.  The cycle cell is the analogue of a vanishing point at cells.
The Game of Cycles resembles the process of converting a graph into a digraph, which can be interpreted in terms of redesigning a road map in a city from two-way streets to one-way streets. The vertices are interpreted as intersections and the edges as streets. The assumption that no sinks and no sources are allowed means that there are no intersections with streets exclusively in or streets exclusively out. 
The roles of the players and the winning strategy can be omitted in this interpretation but another problems becomes more relevant:\ for which game boards (street designs) there exists a traffic design that uses exclusively one way streets. In the game terminology we express it as having no unmarkable edges.


\begin{thm}
\label{prop:cyclecell}
Let $G$ be a finite, connected game board with no sinks and no sources and such that every edge is marked with an arrow. Then $G$ contains a cycle cell.
\end{thm}

Since every edge $e$ of our board will have an arrow, we will henceforth use this notation to indicate the direction of each edge:\ we write $e=uv$ to indicate an edge with an arrow $u \to v$. 

\begin{proof}
Our strategy to show that the game board has a cycle cell is to first find a directed cycle (not necessarily surrounding a cell), and then use an iterative process that produces progressively smaller directed cycles. Since the game board is finite, this will eventually lead to a directed cycle surrounding a single cell---the desired cycle cell. 

It will be important to our construction that any such directed cycle $\gamma$ be \emph{inside-absorbing}:
    if $e$ is an edge in $G$ that lies inside the directed cycle $\gamma$ and is incident to a vertex $v$ on $\gamma$, then $e=uv$ for some vertex $u$. 
(Note that a cycle cell would vacuously satisfy this property.)

\begin{figure}[ht]
		\begin{tikzpicture}[scale=1.2]
		\begin{scope}[very thick, every node/.style={sloped,allow upside down}]
		\Bvertex (A) at (0,0){};
		\Bvertex (B) at (1.3,0){};
		\Bvertex (D) at (1.9,1.2){};
		\Bvertex (E) at (0.8,2.2){};
		\Bvertex (G) at (-0.6,1.6){};
		\Bvertex (H) at (-0.6,.7){};
		\node(I) at (.8,1.4){};
		\node(J) at (0.6,1.4){};
		\node(L) at (1,1.4){};
		\node(M) at (1,1){};
		\node(N) at (1.2,1){};
		\node(P) at (.9,.8){};
		\node(Q) at (.5,.8){};
		\node(R) at (.3,1){};
		\node(S) at (.3,.7){};
		\node(T) at (.2,1.4){};
		\node(U) at (.2,1.8){};
		\node(V) at (.2,1.3){};
		
		\draw (A) -- node {\midarrow} (B);
		\draw (B) -- node {\midarrow} (D);
		\draw (D) -- node {\midarrow} (E);
		\draw (E) -- node {\midarrow} (G);
		\draw (G) -- node {\midarrow} (H);
		\draw (H) -- node {\midarrow} (A);
		
		\draw (I) -- node {\midarrow} (E);
		\draw (L) -- node {\midarrow} (D);
		\draw (M) -- node {\midarrow} (D);
		\draw (P) -- node {\midarrow} (B);
		\draw (Q) -- node {\midarrow} (A);
		\draw (R) -- node {\midarrow} (H);
		\draw (S) -- node {\midarrow} (H);
		\draw (V) -- node {\midarrow} (G);
		\end{scope}
		\end{tikzpicture}
	\caption{An inside-absorbing directed cycle.}
	\label{insideabs}
\end{figure}

In Figure \ref{insideabs}, the boundary of the figure is a directed cycle (since the arrows flow in one direction) and it is inside-absorbing (because on the inside, all incident edges point towards the cycle).

\begin{claim}\label{cl:cycle}
The game board $G$ must contain a directed cycle that is inside-absorbing. 
\end{claim}
\emph{Proof of Claim~\ref{cl:cycle}.}
To obtain our directed cycle we begin at any vertex $u_1$ and construct a directed path $u_1 \to u_2 \to \dots \to u_k$. Since our game board contains no sinks, there will be at least one directed edge $u_1 u_2$ for some vertex $u_2$. To select the other edges in the directed path we define a \emph{clockwise edge selection} rule in the following manner. Select edge $u_2 u_3$ by moving clockwise around vertex $u_2$, starting from the incoming edge $u_1 u_2$ until you encounter the first outgoing edge $u_2 u_3$ for some $u_3$. See Figure~\ref{FBT1} for an example of the rule.

\begin{figure}[ht]
	\captionsetup[subfigure]{labelformat=empty, width=11em}
	\centering
		\subcaptionbox{Clockwise selection rule,\\ first outgoing edge.}{
		\begin{tikzpicture}[scale=1.5]
		\begin{scope}[very thick, every node/.style={sloped,allow upside down}]
		\Bvertex (A) at (0,0)  [label=below right:$u_2$]   {};
		\Bvertex (B) at (0,-1)  [label=right:$u_1$]   {};
		\Bvertex (C) at (-.867,.5)   {};
		\Bvertex (D) at (-.383,.924)  {};
		\Bvertex (E) at (.383,.924) [label=above right:$u_3$]  {};
		\Bvertex (H) at (.957,.29)  {};
		\Bvertex (I) at (.741,.672)  {};
		
		\draw (B) -- node {\midarrow} (A);
		\draw (C) -- node {\midarrow} (A);
		\draw (D) -- node {\midarrow} (A);
		\draw (A) -- node {\midarrow} (E);
		\draw (F) -- node {\midarrow} (A);
		\draw (A) -- node {\midarrow} (I);
		\draw (H)  -- node {\midarrow} (A);
		\draw [-{Triangle[scale=1]}] (-.5,-.8) to[out=160,in=-100] (-1.1,.5) to[out=80,in=140] (.3,1.2);
		\end{scope}		
		\end{tikzpicture}
	}
	\hskip.6in	
	\subcaptionbox{Counterclockwise selection\\ rule, first outgoing edge.}{
		\begin{tikzpicture}[scale=1.5]
		\begin{scope}[very thick, every node/.style={sloped,allow upside down}]
		\Bvertex (A) at (0,0)  [label=below right:$w_2$]   {};
		\Bvertex (B) at (0,-1)  [label=right:$w_1$]   {};
		\Bvertex (C) at (-.867,.5)   {};
		\Bvertex (D) at (-.383,.924)  {};
		\Bvertex (E) at (.383,.924)  {};
		\Bvertex (H) at (.957,.29)  {};
		\Bvertex (I) at (.741,.672) [label=above:$w_3$] {};
		
		\draw (B) -- node {\midarrow} (A);
		\draw (C) -- node {\midarrow} (A);
		\draw (D) -- node {\midarrow} (A);
		\draw (A) -- node {\midarrow} (E);
		\draw (A) -- node {\midarrow} (I);
		\draw (H) -- node {\midarrow} (A);
		\draw [-{Triangle[scale=1]}] (.5,-.8) to[out=20,in=-80] (1.3,.5) to[out=100,in=30] (.9,.9);
		\end{scope}		
		\end{tikzpicture}	
	}
	\hskip.6in
		\subcaptionbox{A path that leads to a\\ counterclockwise cycle.}{
		\begin{tikzpicture}[scale=1.2]
		\begin{scope}[very thick, every node/.style={sloped,allow upside down}]
		\Bvertex (A) at (0,2)   {};
		\Bvertex (B) at (1,1)    {};
		\Bvertex (C) at (1.3,1.8)   {};
		\Bvertex (D) at (2,1)  {};
		\Bvertex (E) at (-1,1)   {};
		\Bvertex (F) at (0,0)    {};
		\Bvertex (G) at (2,0)  {};
		
		\draw (B) -- node {\midarrow} (A);
		\draw (A) -- node {\midarrow} (E);
		\draw (E) -- node {\midarrow} (F);
		\draw (F) -- node {\midarrow} (B);
		\draw (G) -- node {\midarrow} (D);
		\draw (D) -- node {\midarrow} (C);
		\draw (C)  -- node {\midarrow} (A);
		\end{scope}		
		\end{tikzpicture}			
		}
	\caption{Edge selection rules to produce a directed cycle.}
	\label{FBT1}
\end{figure}

Now do the clockwise edge selection rule at $u_3$ to find $u_4$, and do the rule at $u_4$ to get $u_5$, etc., and continue in this fashion until you eventually reach a vertex $u_k$ that already exists in the path.
This yields a directed cycle $C$. It is either a counterclockwise cycle (if it surrounds a bounded region on its left) or a clockwise cycle (if it surrounds a bounded region on its right).

\begin{figure}[H]\centering
	\captionsetup[subfigure]{labelformat=empty, width=10em}
	\subcaptionbox{Clockwise selection rule that produces a counterclockwise cycle.}{
		\begin{tikzpicture}[scale=1.2]
		\begin{scope}[very thick, every node/.style={sloped,allow upside down}]
		\Bvertex (A) at (0,0){};
		\Bvertex (B) at (1.3,0){};
		\Bvertex (C) at (1.9,.7){};
		\Bvertex (D) at (1.9,1.6){};
		\Bvertex (E) at (1.3,2.2){};
		\Bvertex (F) at (0,2.2){};
		\Bvertex (G) at (-0.6,1.6){};
		\Bvertex (H) at (-0.6,.7){};
		\node(I) at (.8,1.4){};
		\node(J) at (0.6,1.4){};
		\node(L) at (1,1.6){};
		\node(M) at (1,1.2){};
		\node(N) at (1.2,1){};
		\node(P) at (.9,.8){};
		\node(Q) at (.5,.8){};
		\node(R) at (.3,1){};
		\node(S) at (.3,.7){};
		\node(T) at (.2,1.4){};
		\node(U) at (.2,1.8){};
		\node(V) at (.2,1.3){};
		\path (0,-.5) node[below,opacity=0] {};
		
		\draw (A) -- node {\midarrow} (B);
		\draw (B) -- node {\midarrow} (C);
		\draw (C) -- node {\midarrow} (D);
		\draw (D) -- node {\midarrow} (E);
		\draw (E) -- node {\midarrow} (F);
		\draw (F) -- node {\midarrow} (G);
		\draw (G) -- node {\midarrow} (H);
		\draw (H) -- node {\midarrow} (A);
		
		\draw (I) -- node {\midarrow} (E);
		\draw (J) -- node {\midarrow} (F);
		\draw (L) -- node {\midarrow} (D);
		\draw (M) -- node {\midarrow} (D);
		\draw (N) -- node {\midarrow} (C);
		\draw (P) -- node {\midarrow} (B);
		\draw (Q) -- node {\midarrow} (A);
		\draw (R) -- node {\midarrow} (H);
		\draw (S) -- node {\midarrow} (H);
		\draw (T) -- node {\midarrow} (H);
		\draw (U) -- node {\midarrow} (G);
		\draw (V) -- node {\midarrow} (G);
		\end{scope}
		\end{tikzpicture}}
	\hskip.25in
	\subcaptionbox{Clockwise selection rule that produces a clockwise cycle.}{
		\begin{tikzpicture}[scale=1.1]
		\begin{scope}[very thick, every node/.style={sloped,allow upside down}]
		\Bvertex (A) at (0,0) [label=above:$a$]{};
		\Bvertex (B) at (1.3,0){};
		\Bvertex (C) at (1.9,.7){};
		\Bvertex (D) at (1.9,1.6){};
		\Bvertex (E) at (1.3,2.2){};
		\Bvertex (F) at (0,2.2){};
		\Bvertex (G) at (-0.6,1.6){};
		\Bvertex (H) at (-0.6,.7) [label=right:$b$]{};
		\node(I) at (1.5,3){};
		\node(J) at (-0.3,3){};
		\node(K) at (2.6,2.2){};
		\node(L) at (2.7,1.6){};
		\node(M) at (2.6,1.2){};
		\node(N) at (2.6,1){};
		\node(O) at (2.6,.4){};
		\node(P) at (1.9,-0.5){};
		\node(Q) at (-.7,-0.5){};
		\node(R) at (-1.5,1.1){};
		\node(S) at (-1.5,.7){};
		\node(T) at (-1.5,.3){};
		\node(U) at (-1.5,2.0){};
		\node(V) at (-1.5,1.4){};
		\path (0,-.6) node[below,opacity=0] {};

		\draw (B) -- node {\midarrow} (A);
		\draw (A) -- node {\midarrow} (H);
		\draw (H) -- node {\midarrow} (G);
		\draw (G) -- node {\midarrow} (F);
		\draw (F) -- node {\midarrow} (E);
		\draw (E) -- node {\midarrow} (D);
		\draw (D) -- node {\midarrow} (C);
		\draw (C) -- node {\midarrow} (B);
		
		\draw (I) -- node {\midarrow} (E);
		\draw (J) -- node {\midarrow} (F);
		\draw (K) -- node {\midarrow} (D);
		\draw (L) -- node {\midarrow} (D);
		\draw (M) -- node {\midarrow} (D);
		\draw (N) -- node {\midarrow} (C);
		\draw (O) -- node {\midarrow} (C);	
		\draw (P) -- node {\midarrow} (B);
		\draw (Q) -- node {\midarrow} (A);
		\draw (R) -- node {\midarrow} (H);
		\draw (S) -- node {\midarrow} (H);
		\draw (T) -- node {\midarrow} (H);
		\draw (U) -- node {\midarrow} (G);
		\draw (V) -- node {\midarrow} (G);
		\end{scope}
		\end{tikzpicture}}
	\hskip.25in
	\subcaptionbox{Counterclockwise selection rule used inside a clockwise cycle.}{
		\begin{tikzpicture}[scale=1.1]
		\begin{scope}[very thick, every node/.style={sloped,allow upside down}]
		\Bvertex (A) at (0,-.2) [label=above:$a$]{};
		\Bvertex (B) at (1.1,-.2){};
		\Bvertex (C) at (1.9,.5){};
		\Bvertex (D) at (1.9,1.6){};
		\Bvertex (E) at (1.1,2.2){};
		\Bvertex (F) at (0,2.2){};
		\Bvertex (G) at (-0.8,1.6){};
		\Bvertex (H) at (-0.8,.5) [label=right:$b$]{};
		\Bvertex (aa) at (1.4,1.4) {};
		\Bvertex (bb) at (.8,1.7) [label=left:$e$]{};
		\Bvertex (cc) at (.4,1.1)[label=right:$d$]{};
		\Bvertex (dd) at (1.4,.8) {};
		\Bvertex (ee) at (.8,.5){};
		\Bvertex (ff) at (-.3,1.2)[label=above:$c$]{};
		
		\node(I) at (1.5,3){};
		\node(J) at (-0.3,3){};
		\node(K) at (2.6,2.2){};
		\node(L) at (2.7,1.6){};
		\node(M) at (2.6,1.2){};
		\node(N) at (2.6,1){};
		\node(O) at (2.6,.4){};
		\node(P) at (1.9,-0.5){};
		\node(Q) at (-.7,-0.5){};
		\node(R) at (-1.5,1.1){};
		\node(S) at (-1.5,.7){};
		\node(T) at (-1.5,.3){};
		\node(U) at (-1.5,2.0){};
		\node(V) at (-1.5,1.4){};
		\node(W) at (-.1,-.9){};
		
		\draw (B) -- node {\midarrow} (A);
		\draw (A) -- node {\midarrow} (H);
		\draw (H) -- node {\midarrow} (G);
		\draw (G) -- node {\midarrow} (F);
		\draw (F) -- node {\midarrow} (E);
		\draw (E) -- node {\midarrow} (D);
		\draw (D) -- node {\midarrow} (C);
		\draw (C) -- node {\midarrow} (B);
		\draw (aa) -- node {\midarrow} (D);
		\draw (bb) -- node {\midarrow} (aa);
		\draw (cc) -- node {\midarrow} (bb);
		\draw (C) -- node {\midarrow} (dd);
		\draw (dd) -- node {\midarrow} (ee);
		\draw (ee) -- node {\midarrow} (cc);
		\draw (ff) -- node {\midarrow} (cc);
		\draw (H) -- node {\midarrow} (ff);
		
		\draw (I) -- node {\midarrow} (E);
		\draw (J) -- node {\midarrow} (F);
		\draw (K) -- node {\midarrow} (D);
		\draw (L) -- node {\midarrow} (D);
		\draw (M) -- node {\midarrow} (D);
		\draw (N) -- node {\midarrow} (C);
		\draw (O) -- node {\midarrow} (C);	
		\draw (P) -- node {\midarrow} (B);
		\draw (Q) -- node {\midarrow} (A);
		\draw (R) -- node {\midarrow} (H);
		\draw (S) -- node {\midarrow} (H);
		\draw (T) -- node {\midarrow} (H);
		\draw (U) -- node {\midarrow} (G);
		\draw (V) -- node {\midarrow} (G);
		\draw (W) -- node {\midarrow} (A);
		\end{scope}
		\end{tikzpicture}}
	\caption{The results of using either the clockwise or counterclockwise  edge selection rules.}
	\label{FBT2}
\end{figure}

Case 1:\ Suppose $C$ is a counterclockwise cycle. Then by construction $C$ is inside-absorbing, because our clockwise edge selection rule at vertices ensures that an inside edge (on the ``left side'' of the cycle) must be directed towards the vertex.  See Figure \ref{FBT2}, leftmost diagram.

Case 2:\ If $C$ is a clockwise cycle, then our construction has produced a cycle with the property that every edge \emph{outside} the cycle that is incident to a cycle vertex must be directed toward that vertex, as in the middle diagram of Figure \ref{FBT2}.  We shall use this $C$ to produce a new cycle that is inside-absorbing. 

Start at any directed edge $w_1 w_2$ of $C$.  
Now use a \emph{counterclockwise edge selection} rule to determine an edge $w_2 w_3$ by moving counterclockwise around vertex $w_2$, starting from the incoming edge $w_1 w_2$ until you encounter the first outgoing edge $w_2 w_3$ for some $w_3$.  (See Figure \ref{FBT1}, middle diagram.) Then do the 
counterclockwise edge selection rule 
at $w_3$ to find $w_4$, and continue in this fashion, repeatedly using counterclockwise edge selection until you eventually reach a vertex $w_k$ that already exists in the path. This yields a directed cycle $C'$. 

For an example, see Figure \ref{FBT2}. Performing the counterclockwise edge selection rule on the clockwise cycle in the middle diagram, starting at the edge $ab$, might produce an inner path like $a \to b \to c \to d \to e \to ...$  evident in the rightmost diagram, which winds around and returns to $d$, yielding a directed cycle $C'$.

Note that $C'$ must be contained inside $C$, because all edges outside $C$ and incident to $C$ must be directed towards $C$. This prevents $C'$ from leaving $C$. Thus $C'$ cannot contain more cells than $C$. 

If $C'$ and $C$ contain the same number of cells, then $C'=C$. This means that the clockwise and counterclockwise edge selection rules selected the same edges, implying that all other edges incident to the directed cycle were directed towards cycle vertices, and so Claim~\ref{cl:cycle} is satisfied.

So suppose instead that $C'$ contains strictly fewer cells than $C$. Then we must again break into cases depending on whether $C'$ is clockwise or counterclockwise. If $C'$ is clockwise then (by a similar reasoning as in Case 1 above) $C'$ also satisfies the claim. 
If instead $C'$ is counterclockwise, then 
in a similar fashion to Case 2 above, we can produce another cycle $C''$, contained in $C'$, this time using a \emph{clockwise} edge selection rule.

We can continue producing smaller directed cycles using the clockwise or counterclockwise edges selection rules until an inside-absorbing directed cycle has been created. This process will terminate because at each iteration of creating a new directed cycle, the number of cells within the directed cycle either gets smaller or stays the same. This number cannot get smaller forever, so it must stabilize at some point. This completes the proof of Claim~\ref{cl:cycle}.

\begin{claim}\label{cl:cycle2}
Every inside-absorbing directed cycle on game board $G$ either borders a cycle cell or contains a strictly smaller inside-absorbing directed cycle. 
\end{claim}

\emph{Proof of Claim~\ref{cl:cycle2}.}
Consider the complex consisting of the directed cycle and all edges, vertices, and cells inside of it. Remove the interior of all cells (0-, 1-, and 2-dimensional) in this complex that intersect the directed cycle. 
Call the resulting object the \emph{gut} of the cycle. Note that the gut may take on a number of forms, and may even be disconnected. Some examples are shown in Figure~\ref{FBT3} below. The shading inside the gut of the figure on the left is meant to indicate the existence of any number of configurations of edges and vertices making up the interior of the gut.

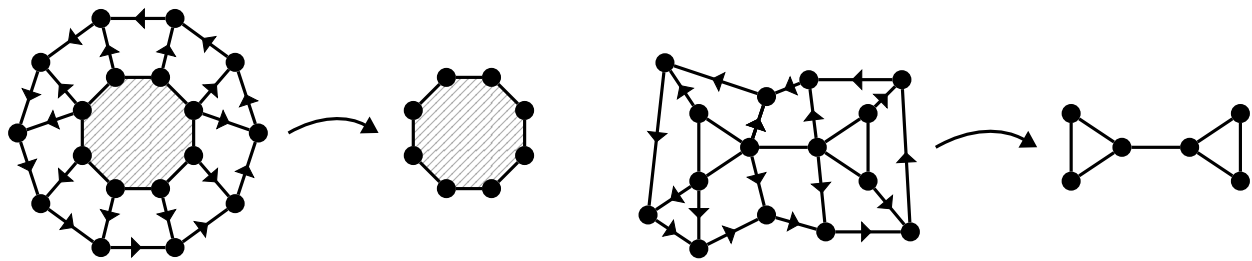
\begin{figure}[H]
\centering
\begin{tikzpicture}[scale=.4]
\begin{scope}[very thick, every node/.style={sloped,allow upside down}]
\Bvertex (A) at (-1.23,-3.81){};
\Bvertex (B) at (1.23,-3.81){};
\Bvertex (C) at (3.23,-2.35){};
\Bvertex (D) at (4,0){};
\Bvertex (E) at (3.23,2.35){};
\Bvertex (F) at (1.23,3.81){};
\Bvertex (G) at (-1.23,3.81){};
\Bvertex (H) at (-3.23,2.35){};
\Bvertex (I) at (-4,0){};
\Bvertex (J) at (-3.23,-2.35){};

\Bvertex (A_1) at (-0.75,-1.85){};
\Bvertex (B_1) at (0.75,-1.85){};
\Bvertex (C_1) at (1.85,-0.75){};
\Bvertex (D_1) at (1.85,0.75){};
\Bvertex (E_1) at (0.75,1.85){};
\Bvertex (F_1) at (-0.75,1.85){};
\Bvertex (G_1) at (-1.85,0.75){};
\Bvertex (H_1) at (-1.85,-0.75){};

\Bvertex (A_2) at (10.25,-1.85){};
\Bvertex (B_2) at (11.75,-1.85){};
\Bvertex (C_2) at (12.85,-0.75){};
\Bvertex (D_2) at (12.85,0.75){};
\Bvertex (E_2) at (11.75,1.85){};
\Bvertex (F_2) at (10.25,1.85){};
\Bvertex (G_2) at (9.15,0.75){};
\Bvertex (H_2) at (9.15,-0.75){};

\draw (A) -- node {\midarrow} (B);
\draw (B) -- node {\midarrow} (C);
\draw (C) -- node {\midarrow} (D);
\draw (D) -- node {\midarrow} (E);
\draw (E) -- node {\midarrow} (F);
\draw (F) -- node {\midarrow} (G);
\draw (G) -- node {\midarrow} (H);
\draw (H) -- node {\midarrow} (I);
\draw (I) -- node {\midarrow} (J);
\draw (J) -- node {\midarrow} (A);

\draw (A_1) to (B_1);
\draw (B_1) to (C_1);
\draw (C_1) to (D_1);
\draw (D_1) to (E_1);
\draw (E_1) to (F_1);
\draw (F_1) to (G_1);
\draw (G_1) to (H_1);
\draw (H_1) to (A_1);

\draw (A_1) -- node {\midarrow} (A);
\draw (B_1) -- node {\midarrow} (B);
\draw (C_1) -- node {\midarrow} (C);
\draw (D_1) -- node {\midarrow} (D);
\draw (D_1) -- node {\midarrow} (E);
\draw (E_1) -- node {\midarrow} (F);
\draw (F_1) -- node {\midarrow} (G);
\draw (G_1) -- node {\midarrow} (H);
\draw (G_1) -- node {\midarrow} (I);
\draw (H_1) -- node {\midarrow} (J);

\draw [bend left,-{Triangle[scale=1]}]  (5,0) to  (8,0);

\draw (A_2) to (B_2);
\draw (B_2) to (C_2);
\draw (C_2) to (D_2);
\draw (D_2) to (E_2);
\draw (E_2) to (F_2);
\draw (F_2) to (G_2);
\draw (G_2) to (H_2);
\draw (H_2) to (A_2);

\begin{scope}[on background layer]
\path[pattern color = gray, pattern = north east lines, opacity=0.8] (A_2.center) -- (B_2.center) -- (C_2.center) -- (D_2.center) -- (E_2.center) -- (F_2.center) -- (G_2.center) -- (H_2.center) -- cycle;
\end{scope}

\begin{scope}[on background layer]
\path[pattern color = gray, pattern = north east lines, opacity=0.8] (A_1.center) -- (B_1.center) -- (C_1.center) -- (D_1.center) -- (E_1.center) -- (F_1.center) -- (G_1.center) -- (H_1.center) -- cycle;
\end{scope}
\end{scope}
\end{tikzpicture}	
\hfill
\begin{tikzpicture}[scale=.45]
\begin{scope}[very thick, every node/.style={sloped,allow upside down}]
\Bvertex (A) at (-1,0){};
\Bvertex (B) at (1,0){};
\Bvertex (C) at (-2.5,1){};
\Bvertex (D) at (-2.5,-1){};
\Bvertex (E) at (2.5,-1){};
\Bvertex (F) at (2.5,1){};

\Bvertex (A_1) at (-.5,1.5){};
\Bvertex (A_2) at (-.5,-2){};
\Bvertex (B_1) at (.75,2){};
\Bvertex (B_2) at (1.25,-2.5){};
\Bvertex (C_1) at (-3.5,2.5){};
\Bvertex (D_1) at (-4,-2){};
\Bvertex (D_2) at (-2.5,-3){};
\Bvertex (F_1) at (3.5,2){};
\Bvertex (E_1) at (3.75,-2.5){};

\Bvertex (A_3) at (10,0){};
\Bvertex (B_3) at (12,0){};
\Bvertex (C_3) at (8.5,1){};
\Bvertex (D_3) at (8.5,-1){};
\Bvertex (E_3) at (13.5,-1){};
\Bvertex (F_3) at (13.5,1){};

\draw (A) to (B);
\draw (A) to (C);
\draw (A) to (D);
\draw (C) to (D);
\draw (B) to (F);
\draw (B) to (E);
\draw (F) to (E);

\draw (D) -- node {\midarrow} (D_1);
\draw (D) -- node {\midarrow} (D_2);
\draw (D_1) -- node {\midarrow} (D_2);
\draw (C) -- node {\midarrow} (C_1);
\draw (C_1) -- node {\midarrow} (D_1);
\draw (A_1) -- node {\midarrow} (C_1);
\draw (A) -- node {\midarrow} (A_1);
\draw (A) -- node {\midarrow} (A_2);
\draw (D_2) -- node {\midarrow} (A_2);
\draw (B) -- node {\midarrow} (B_1);
\draw (B) -- node {\midarrow} (B_2);
\draw (A) -- node {\midarrow} (A_1);
\draw (A_2) -- node {\midarrow} (B_2);
\draw (A) -- node {\midarrow} (A_1);
\draw (B_1) -- node {\midarrow} (A_1);
\draw (F_1) -- node {\midarrow} (B_1);
\draw (F) -- node {\midarrow} (F_1);
\draw (E_1) -- node {\midarrow} (F_1);
\draw (E) -- node {\midarrow} (E_1);
\draw (B_2) -- node {\midarrow} (E_1);

\draw [bend left,-{Triangle[scale=1]}]  (4.5,0) to  (7.5,0);

\draw (A_3) to (B_3);
\draw (A_3) to (C_3);
\draw (A_3) to (D_3);
\draw (C_3) to (D_3);
\draw (B_3) to (F_3);
\draw (B_3) to (E_3);
\draw (F_3) to (E_3);
\end{scope}
\end{tikzpicture}	

\caption{The results of removing the interior of all cells intersecting a directed cycle.}\label{FBT3}
\end{figure}

If the gut is empty, then there are no vertices lying inside the directed cycle, hence the directed cycle must be the border of a cycle cell.  

If the gut is non-empty, consider any vertex $v$ in the gut. Because $v$ is not a source, and because all edges outside the gut and incident to it are pointing away from it, there must be an edge inside the gut that points towards $v$.  
Since the same is true of every vertex in the gut, we can trace 
\textit{backwards} along directed edges from a vertex $v$ to produce a path that must remain inside the gut. As in the proof of Claim~\ref{cl:cycle}, we can use a clockwise selection rule to choose such a path.  Since this path must eventually intersect itself we will locate a directed cycle in inside the gut. If it is not inside-absorbing, we can argue as in our proof of Claim~\ref{cl:cycle} that by alternating our selection rule between clockwise and counterclockwise as needed, we can eventually obtain a directed cycle that is inside-absorbing and smaller than the original directed cycle.
This completes the proof of Claim~\ref{cl:cycle2}.

\medskip
The two claims above allow us to produce a sequence of progressively smaller directed-cycles within game board $G$. By finiteness of the game board, this sequence must be finite and thus ends when a cycle cell has been produced, the desired conclusion of proof of Theorem~\ref{prop:cyclecell}.
\end{proof}

In the proof of Claim \ref{cl:cycle2}, it may be worth noting that the gut 
cannot be entirely one-dimensional, 
i.e., a collection of trees. The proof shows why, because tracing backwards in the gut must yield a cycle in the gut.
So the gut of the diagram in Figure \ref{1dgut} cannot occur.

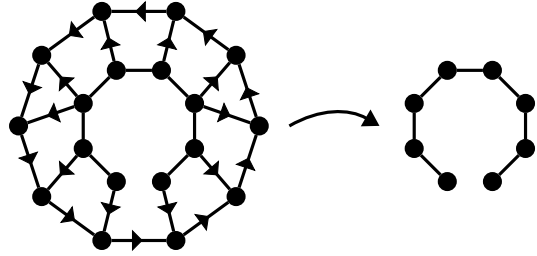
\begin{figure}[H]\center
\begin{tikzpicture}[scale=.4]
\begin{scope}[very thick, every node/.style={sloped,allow upside down}]
\Bvertex (A) at (-1.23,-3.81){};
\Bvertex (B) at (1.23,-3.81){};
\Bvertex (C) at (3.23,-2.35){};
\Bvertex (D) at (4,0){};
\Bvertex (E) at (3.23,2.35){};
\Bvertex (F) at (1.23,3.81){};
\Bvertex (G) at (-1.23,3.81){};
\Bvertex (H) at (-3.23,2.35){};
\Bvertex (I) at (-4,0){};
\Bvertex (J) at (-3.23,-2.35){};

\Bvertex (A_1) at (-0.75,-1.85){};
\Bvertex (B_1) at (0.75,-1.85){};
\Bvertex (C_1) at (1.85,-0.75){};
\Bvertex (D_1) at (1.85,0.75){};
\Bvertex (E_1) at (0.75,1.85){};
\Bvertex (F_1) at (-0.75,1.85){};
\Bvertex (G_1) at (-1.85,0.75){};
\Bvertex (H_1) at (-1.85,-0.75){};

\Bvertex (A_2) at (10.25,-1.85){};
\Bvertex (B_2) at (11.75,-1.85){};
\Bvertex (C_2) at (12.85,-0.75){};
\Bvertex (D_2) at (12.85,0.75){};
\Bvertex (E_2) at (11.75,1.85){};
\Bvertex (F_2) at (10.25,1.85){};
\Bvertex (G_2) at (9.15,0.75){};
\Bvertex (H_2) at (9.15,-0.75){};


\draw (A) -- node {\midarrow} (B);
\draw (B) -- node {\midarrow} (C);
\draw (C) -- node {\midarrow} (D);
\draw (D) -- node {\midarrow} (E);
\draw (E) -- node {\midarrow} (F);
\draw (F) -- node {\midarrow} (G);
\draw (G) -- node {\midarrow} (H);
\draw (H) -- node {\midarrow} (I);
\draw (I) -- node {\midarrow} (J);
\draw (J) -- node {\midarrow} (A);

\draw (B_1) to (C_1);
\draw (C_1) to (D_1);
\draw (D_1) to (E_1);
\draw (E_1) to (F_1);
\draw (F_1) to (G_1);
\draw (G_1) to (H_1);
\draw (H_1) to (A_1);

\draw (A_1) -- node {\midarrow} (A);
\draw (B_1) -- node {\midarrow} (B);
\draw (C_1) -- node {\midarrow} (C);
\draw (D_1) -- node {\midarrow} (D);
\draw (D_1) -- node {\midarrow} (E);
\draw (E_1) -- node {\midarrow} (F);
\draw (F_1) -- node {\midarrow} (G);
\draw (G_1) -- node {\midarrow} (H);
\draw (G_1) -- node {\midarrow} (I);
\draw (H_1) -- node {\midarrow} (J);

\draw [bend left,-{Triangle[scale=1]}]  (5,0) to  (8,0);

\draw (B_2) to (C_2);
\draw (C_2) to (D_2);
\draw (D_2) to (E_2);
\draw (E_2) to (F_2);
\draw (F_2) to (G_2);
\draw (G_2) to (H_2);
\draw (H_2) to (A_2);
\end{scope}
\end{tikzpicture}
\caption{An example of a gut which violates the no source rule if directions were added to its edges.}
\label{1dgut}
\end{figure}


\section{Further Questions}

There are many unanswered questions that naturally emerge from our exploration.  For instance:
\begin{itemize}
    \item Can you determine winning strategies for other classes of boards?  For instance, is there a general winning strategy for boards with exactly 2 cells? Such a result would extend Theorems \ref{ngonchordthm} and  \ref{ngonwithflap}. The board formed by gluing a heptagon and a pentagon together along two adjacent common edges is not covered by those theorems.
    \item To augment Theorem \ref{invsym}, what can be said about boards that have reflective symmetry but with more than one self-involutive edge, as in Figure \ref{fig:symmeg}?  This figure is the simplest board that is not covered by any of our theorems.
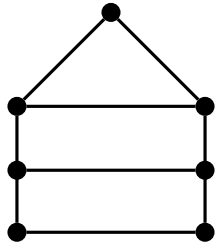
\begin{figure}[H]
    \centering
    \begin{tikzpicture}[scale=1.25]
		\begin{scope}[very thick, every node/.style={sloped,allow upside down}]
		\Bvertex (A) at (0,2)  [] {};
		\Bvertex (B) at (1,1)  [] {};
		\Bvertex (E) at (-1,1) []{};
		\Bvertex (F) at (-1,-.34) [] {};
        \Bvertex (G) at (1,-.34) [] {};
        \Bvertex (H) at (-1,.32) [] {};
        \Bvertex (I) at (1,.32) [] {};
        
		\draw (A) to (B);
		\draw (A) to (E);
		\draw (B) to (G);
		\draw (F) to (E);
		\draw (F) to (G);
		\draw (H) to (I);
		\draw (E) to (B);
		\end{scope}		
    \end{tikzpicture}
    \caption{A board with multiple fixed edges under reflective symmetry.}
    \label{fig:symmeg}
\end{figure}
    \item All the theorems we have proved so far have shown, for various classes of boards, that if the number of edges in the board is odd, Player~1 has a winning strategy, and otherwise Player~2 has a winning strategy. Is there a board that does not follow this pattern?
    \item How does the game change in strategy if finishing the game without a cycle cell is considered a ``draw'' (rather than a win for the last player)?
    \item On a larger board, how does the game change in strategy if the object were to complete as many cycle cells as possible (rather than just one)? (One might explore rewarding the completion of a cycle cell with another turn.)
    \item How would you play the game with 3 or more players?  Are there interesting game boards that can be analyzed?
    \item This game is played on a 2-dimensional cellular complex that can be embedded in the plane. One could define a similar game for a $3$-dimensional complex, where $2$-dimensional faces are marked with an orientation.
\end{itemize}

There are plenty more avenues for exploration here.  Enjoy the game!

\bibliographystyle{plain}
\bibliography{references} 

\vfill

\end{document}